\documentclass[a4paper,11pt]{amsart}

\usepackage{amsmath}
\usepackage{amsfonts}
\usepackage{amssymb}
\usepackage{latexsym}

\usepackage[T1]{fontenc}
\usepackage[utf8]{inputenc}
\usepackage[
textwidth=15cm,
textheight=24cm,
hmarginratio=1:1,
vmarginratio=1:1]{geometry}

\usepackage{amssymb}

\allowdisplaybreaks
\newcommand{\Cn}{\mathbb{C}^n}

\newcommand{\Sn}{\mathbb{S}^n}

\newcommand{\Bn}{\mathbb{B}^n}
\newcommand{\D}{\mathbb D}

\newcommand{\C}{\mathbb C}

\newcommand{\B}{\mathbb B}
\newcommand{\Sp}{\mathbb S}
\newcommand{\BMOA}{\mathit{BMOA}}
\newcommand{\VMOA}{\mathit{VMOA}}

\newtheorem{theorem}{Theorem}
\newtheorem{lemma}[theorem]{Lemma}
\newtheorem{prop}[theorem]{Proposition}

\theoremstyle{remark}

\newtheorem{remark}[theorem]{Remark}

\numberwithin{equation}{section}
\numberwithin{theorem}{section}

\title[Rigidity of Volterra-type integral operators on Hardy spaces]{Rigidity of Volterra-type integral operators on  Hardy spaces of the unit ball}

\author[Santeri Miihkinen]{Santeri Miihkinen}
\address{Santeri Miihkinen \\Department of Mathematics \\
\AA bo Akademi University\\
FI-20500 \AA bo\\
Finland} \email{santeri.miihkinen@abo.fi}

\author[Jordi Pau]{Jordi Pau}
\address{Jordi Pau \\Departament de Matem\`{a}tiques i Inform\`{a}tica \\
Universitat de Barcelona\\
08007 Barcelona\\
Catalonia, Spain} \email{jordi.pau@ub.edu}

\author[Antti Per\"{a}l\"{a}]{Antti Per\"{a}l\"{a}}
\address{Antti Per\"{a}l\"{a} \\Department of Mathematical Sciences \\
 Chalmers University of Technology and the University of Gothenburg \\
Gothenburg SE-412 96 \\
Sweden} \email{perala.math@gmail.com}

\author[Maofa Wang]{Maofa Wang}
\address{Maofa Wang \\School of Mathematics and Statistics, Wuhan University, Wuhan
430072 \\  China} \email{mfwang.math@whu.edu.cn}



\thanks{S. Miihkinen
was supported by the
Academy of Finland project 296718. J. Pau was partially
 supported by  the grants MTM2017-83499-P (Ministerio de Educaci\'{o}n y Ciencia) and 2017SGR358 (Generalitat de Catalunya). A. Per\"al\"a acknowledges financial support from the Spanish Ministry of Economy and Competitiveness, through the Mar\'ia de Maeztu Programme for Units of Excellence in R\&D (MDM-2014-0445).  M. Wang was partially supported by  {NSFC (11771340)}  and China Scholarship Council (201806275003).}

\begin{document}

\begin{abstract}
We
establish that  the  Volterra-type integral operator $J_b$ on the Hardy spaces $H^p$ of the unit ball $\Bn$  exhibits
a rather strong rigid behavior. More precisely, we show that the compactness, strict singularity and $\ell^p$-singularity of $J_b$ are equivalent on $H^p$ for any $1 \le p < \infty$. Moreover,  we show that the operator $J_b$ acting on $H^p$ cannot fix an isomorphic copy of $\ell^2$ when $p \ne 2.$
\end{abstract}

\maketitle

\section{Introduction}

An operator $T\colon X \to Y$ between Banach spaces $X$ and $Y$ is \textit{strictly singular} if its restriction to any infinite-dimensional subspace $M$ of $X$ is not a linear isomorphism onto its range, i.e. the restriction is not bounded below on $M.$ This class of operators forms a two-sided operator ideal and was introduced by T. Kato \cite{Kat} in connection with the perturbation theory of Fredholm operators.  If $T$ is not bounded below on any subspace $M$ isomorphic to the sequence space $\ell^p,$ then $T$ is said to be $\ell^p$-\textit{singular}. These notions generalize the concept of compact operators. Examples of strictly singular non-compact operators are the inclusion mappings $i_{p,q}:\ell^p \hookrightarrow \ell^q$, where $1\le p<q<\infty$. The following inclusions hold: $K(X) \subset S(X) \subset S_p(X),$ where $K(X)$ is the class of compact operators on $X$, $S(X)$ the class of strictly singular operators and $S_p(X)$ the class of $\ell^p-$singular operators on $X.$ In general, these classes are distinct, but coincide e.g.~in the case of $X$ being a Hilbert space, see \cite[Chapter 5]{P1980}.

The purpose of this paper is to study the strict singularity of the Volterra-type integration operator $J_ b$ acting on the Hardy spaces of the unit ball $\Bn$, extending the results previously obtained in \cite{M1,MNST} for the case of the unit disk $\D$. For a holomorphic function $b$ on $\Bn$, the operator $J_ b$ is defined as
\[
 J_ b f(z)=\int_{0}^{1} f(tz) Rb(tz) \frac{dt}{t},\qquad z\in \Bn
\]
for $f$ holomorphic on $\Bn$. Here $Rb$ denotes the radial derivative of $b$, that is,
\[
Rb(z)= \sum_{k=1}^{n} z_ k \frac{\partial b}{\partial z_ k} (z),\qquad z=(z_ 1,\dots,z_ n)\in \Bn.
\]
It is well-known \cite{AC, AS0, Pau2016} that $J_ b$ is bounded on the Hardy space $H^p(\Bn)$ if and only if $b\in BMOA(\Bn)$, the space of all holomorphic functions of bounded mean oscillation; and $J_ b$ is compact on $H^p(\Bn)$ if and only if $b\in VMOA(\Bn)$, the space of holomorphic functions on $\Bn$ of vanishing mean oscillation. For $0<p<\infty$, the Hardy space $H^p:=H^p(\Bn)$ consists of those holomorphic functions $f$ on $\Bn$ with
 \[ \|f\|_{H^p}^p=\sup_{0<r<1}\int_{\Sn} \!\! |f(r\zeta)|^p \,d\sigma(\zeta)<\infty,\]
where $d\sigma$ is the surface measure on the unit sphere $\Sn:=\partial \Bn$ normalized so that $\sigma(\Sn)=1$.  The mentioned operator $J_ b$ became extremely popular in recent years, being studied in many spaces of holomorphic functions (see \cite{Aleman06, MPPW, Pau2016, Sis} and the references therein). As far as we know, the generalization of the operator $J_ b$ acting on holomorphic function spaces of the unit ball of $\Cn$, as defined here, was introduced by Z. Hu \cite{Hu}. A fundamental property of the Volterra integration operator $J_ b$ is the basic identity
\[
R(J_ b f)(z)=f(z)\,Rb(z),\quad z\in \Bn.
\]
The existence of non-compact strictly singular operators acting on the Hardy space $H^p(\D)$ for $p \ne 2$  can be seen by considering the inclusion mappings between the sequence spaces $\ell^p, \, p \ne 2$, and $\ell^2$ and utilizing the fact that $H^p(\D)$ contains  complemented copies of $\ell^p$ and $\ell^2.$ The existence of such operators can be  transferred to the case of the Hardy spaces $H^p(\B^n), \, 1 \le p < \infty,$ since they are all isomorphic to $H^p(\D)$ by the result of Wojtaszczyk \cite{Woj1983}.

We recall that, for a Banach space $X$, a bounded linear operator $T \colon X \rightarrow X$  is said to fix a copy of a given Banach space $E$, if there is a closed subspace $M\subset X$, linearly isomorphic to $E$, and $c > 0$ so that $\|Tx\| \ge c\|x\|$ for all $x \in M$ (that is, the restriction $T_{|M}$ defines an isomorphism onto its range). Our first result, proved in the case of the unit disk in \cite{M1}, shows that $J_b$ is compact on $H^p$ if and only if it is strictly singular; as when it is not compact it fixes an isomorphic copy of $\ell^p$ inside $H^p$.

\begin{theorem}
\label{T_gstrictsing}
Let $b \in BMOA(\Bn) \setminus VMOA(\Bn)$ and $1 \le p < \infty.$ Then the operator $J_b \colon H^p \to H^p$ fixes an isomorphic copy of $\ell^p$ inside $H^p.$
\end{theorem}
As a consequence, for $b\in BMOA(\Bn)\setminus VMOA(\Bn)$, the operator $J_b$ is not $\ell^p$-singular. Hence the notions of compactness, strict singularity and $\ell^p$-singularity coincide in the case of $J_b$ acting on $H^p.$

Theorem \ref{T_gstrictsing} is established in a similar manner as in the one-dimensional case, by constructing bounded operators $V\colon \ell^p \to H^p$ and $U\colon \ell^p \to H^p$ such that $U = J_bV$, where $V(\ell^p) = M$ is the  closed linear span of suitably chosen test functions $f_{a_k} \in H^p$ and the operator $U$ is an isomorphism onto its range $U(\ell^p) = J_b(M)$.\\


Our second main result extends the one obtained in \cite{MNST} to the setting of the unit ball. The proof requires different techniques, as some of the tools utilized in \cite{MNST} are not available or useful in higher dimensions such as the Riemann mapping theorem. We utilize different equivalent norms and Carleson measures among other techniques. It is interesting to contrast this result with the one obtained in \cite{LNST} for composition operators acting on the Hardy spaces of the unit disk, since composition operators do not exhibit as rigid behavior in regard to $\ell^2$-singularity as the operators $J_b$.

\begin{theorem}\label{mt2}
Let $b  \in BMOA(\Bn)$ and $1 \le p < \infty.$ If there exists a closed infinite-dimensional subspace $M \subset H^p$ such that $J_b\colon H^p \to H^p$ is bounded below on $M,$ then there exists a subspace $N \subset M$ isomorphic to $\ell^p.$ In particular, the operator $J_b$ acting on $H^p$ cannot fix an isomorphic copy of $\ell^2$, i.e.~it is $\ell^2$-singular when $p \ne 2.$
\end{theorem}

We use some standard notation. For any two points $z=(z_ 1,\dots,z_ n)$ and $w=(w_ 1,\dots,w_ n)$ in $\Cn$, we write
$\langle z,w\rangle =z_ 1\bar{w}_ 1+\dots +z_ n \bar{w}_ n,$
 and
 $|z|=\sqrt{\langle z,z\rangle}.$ Typically constants are used with no attempt to calculate their exact values. Given two positive quantities $A$ and $B$, depending on some parameters, we write $A\lesssim B$ to mean that there exists some inessential constant $C>0$ so that $A\leq C B$. The relation $A\gtrsim B$ is defined in an analogous way, and $A \asymp B$ means that both $A\lesssim B$ and $A\gtrsim B$  hold.

The paper is organized as follows. In  Section \ref{sp}, we provide an auxiliary result needed to establish Theorem \ref{T_gstrictsing} in Section \ref{s2}; and  Section \ref{s3} covers our second main result Theorem \ref{mt2}  for which Lemmas \ref{le_vcarleson}-\ref{bddbelU} are crucial tools.

\section{Preliminaries}\label{sp}
It is well known that any function in $H^p$ has radial limits $f(\zeta)=\lim_{r\to 1^{-}} f(r\zeta)$ for a.e. $\zeta \in \Sn$, and $\|f\|^p_{H^p}=\int_{\Sn}|f|^p d\sigma$.
For each $a\in \Bn$, consider the test function
\[
f_a(z) = \frac{(1-|a|^2)^{1/p}}{(1- \langle z, a \rangle)^{(n+1)/p}},\qquad z\in \Bn.
\]
It is easy to see that $f_ a \in H^p$ with $\|f_ a\|_{H^p}\asymp1$, and its radial derivative is given by
\[
Rf_ a(z)=\frac{(n+1)}{p} \,(1-|a|^2)^{1/p}\frac{\langle z,a \rangle }{(1-\langle z,a \rangle )^{\frac{n+1+p}{p}}},\qquad z\in \Bn.
\]
 We need the following lemma regarding their values and the values of $J_b f_a$ peaking in certain subsets of $\Sp^n$.

\begin{lemma}
\label{testfunctions}
Let $b \in \BMOA(\Bn), \, 1 \le p < \infty,$  and $(a_k) \subset \B^n$ be a sequence such that $a_k \to \omega \in \Sp^n$. Define a non-isotropic metric ball $$S_\varepsilon(\omega) = \{z \in \mathbb{S}^n: |1 - \langle z, \omega \rangle| < \varepsilon\}$$ for each $\varepsilon  > 0$. Then
\[
\begin{split}
&\textrm{(i) } \lim_{k  \to \infty}\int_{\Sp^n \setminus S_\varepsilon(\omega)} |f_{a_k}|^p \, d\sigma = 0 \textrm{ for each $\varepsilon  >0$;} \\
&\textrm{(ii) } \lim_{\varepsilon \to 0}\int_{S_\varepsilon(\omega)} |f_{a_k}|^p \, d\sigma = 0 \textrm{ for each $k$;} \\
&(iii) \lim_{k \to \infty} \int_{\Sp^n \setminus S_\varepsilon(\omega)} |J_bf_{a_k}|^p \, d\sigma = 0 \textrm{ for each $\varepsilon > 0$;} \\
&(iv) \lim_{\varepsilon \to 0} \int_{S_\varepsilon(\omega)} |J_b f_{a_k}|^p \, d\sigma = 0 \textrm{ for each k.}
\end{split}
\]
\end{lemma}
\begin{proof}
Fix $\varepsilon > 0$ and $u_0 \in \Sp^n \setminus S_{\varepsilon}(\omega)$.
Now $|1 - \langle u_0, \omega\rangle| \ge \varepsilon$.
Thus it holds that $|1 - \langle u_0, a_k\rangle| \ge \varepsilon/2$ for $k$ large enough.
So we may choose $\delta = \delta(\varepsilon) > 0$ so that $|1 - \langle u_0, a_k\rangle| \ge \delta$ for all $k$ large enough and all $u_0 \in \Sp^n \setminus S_{\varepsilon}(\omega)$ and the condition \textbf{(i)} follows. The proof of \textbf{(ii)} follows from the absolute continuity of the measures $A \mapsto \int_A |f_{a_k}|^p \, d\sigma.$

\textbf{(iii)} Let now $0 < \varepsilon < 1/2.$ We may assume that $b(0)=0$. We first confirm that there exists $\delta = \delta(\varepsilon) > 0$ such that $|1 - \langle ru_0, a_k \rangle| \ge \delta$ for all $k$ large enough and all $u_0 \in \Sp^n \setminus S_{\varepsilon}(\omega)$ and $0 \le r \le 1$.
Fix $u_0 \in \Sp^n \setminus S_{\varepsilon}(\omega)$ and suppose that $0 \le r \le 1 - \varepsilon^2.$ Then $$|1 - \langle ru_0, \omega\rangle| \ge 1 - r|\langle u_0, \omega \rangle| \ge 1-(1-\varepsilon^2) = \varepsilon^2.$$ Hence $|1 - \langle ru_0, a_k \rangle| \ge \frac{\varepsilon^2}{2}$ for $k \ge k_0$ for some $k_0 = k_0(\varepsilon)$ and all $0 \le r \le 1-\varepsilon^2$.
Consider then the case $1-\varepsilon^2 < r \le 1.$ Now $$|1 - \langle ru_0, \omega\rangle| = |r- \langle ru_0, \omega\rangle + 1 - r| \ge r |1-\langle u_0, \omega\rangle| - (1-r) \ge r\varepsilon - \varepsilon^2 = \varepsilon(r-\varepsilon) > \varepsilon (1/2 - \varepsilon).$$ Thus it holds that $$|1 - \langle ru_0, a_k\rangle| \ge \varepsilon/2 (1/2 - \varepsilon)$$ for all $k$ large enough and all $1-\varepsilon^2 < r \le 1.$ So we may choose $\delta = \delta(\varepsilon) > 0$ so that $|1 - \langle ru_0, a_k\rangle| \ge \delta$ for all $k$ large enough and all $u_0 \in \Sp^n \setminus S_{\varepsilon}(\omega)$ and $0 \le r \le 1$. For those $u_0$ and $r$ we obtain the estimates
\[
|f_{a_k}(ru_0)|^p = \frac{1-|a_k|^2}{|1- \langle ru_0, a_k \rangle|^{n+1}} \le \frac{1-|a_k|^2}{\delta^{n+1}}
\]
 and
\[
\begin{split}
|Rf_{a_k}(ru_0)|^p &= \left (\frac{n+1}{p}\right )^p \,\frac{(1-|a_ k|^2)\,|\langle ru_ 0,a_ k \rangle |^p}{|1-\langle ru_ 0,a_ k \rangle |^{n+1+p}}
\\
&\le  \frac{C r^p(1-|a_k|^2)}{|1-\langle ru_0, a_k \rangle|^{n + 1 + p}} \le  \frac{Cr^p\,(1-|a_k|^2)}{\delta^{n+1+p}},
\end{split}
\]
for all $k$ large enough, where $C = C(n,p) > 0.$ Now, for a.e. $\zeta \in \Sp^n\setminus S_\varepsilon(\omega)$, we obtain
\[
\begin{split}
|J_bf_{a_k}(\zeta)|^p &= \lim_{r\rightarrow 1^{-}}\left| \int_0^1 f_{a_k}(t r \zeta) Rb(t r \zeta)\frac{dt}{t}\right|^p
\\
&\le C\left(|f_{a_k}(\zeta) b(\zeta)|^p + \left(\int_0^1 |Rf_{a_k}(t\zeta)b(t\zeta)|\frac{dt}{t}\right)^p \right)
\\
&\le C \left(\frac{1-|a_k|^2}{\delta^{n+1}} |b(\zeta)|^p + \frac{1-|a_k|^2}{\delta^{n+1+p}}\left(\int_0^1 |b(t\zeta)|\, dt\right)^p\right),
\end{split}
\]
where constants may depend on $n$ and $p$. Utilizing the well known pointwise estimate $|b(z)| \lesssim \|b\|_{BMOA} \log \frac{1}{1-|z|}$ for $z \in \B^n$, where and $\|b\|_{BMOA}$ is the $\BMOA$-seminorm, we have $\int_0^1 |b(t\zeta)|\, dt \le C \|b\|_{BMOA}.$ Therefore $$\int_{\Sp^n \setminus S_\varepsilon(\omega)}|J_bf_{a_k}(\zeta)|^p \, d\sigma(\zeta) \le C \left( \frac{1-|a_k|^2}{\delta^{n+1}} \|b\|_{H^p}^p + \frac{1-|a_k|^2}{\delta^{n+1+p}}\|b\|_{BMOA}^p\right) \to 0,$$ as $k \to \infty,$ since $\|b\|_{H^p} \lesssim \|b\|_{BMOA}< \infty.$

\textbf{(iv)} If $k$ is fixed, the claim follows from the absolute continuity of the measure $A \mapsto \int_{A}|J_b f_{a_k}|^p \, d\sigma$.
\end{proof}

\section{$\ell^p-$singularity of $J_b$}\label{s2}
In this section, we establish the fact that a non-compact integration operator $J_b$ acting on $H^p$ fixes a copy of $\ell^p$. We begin with an auxiliary result.
\begin{prop}
\label{bddmap}
Let $1 \le p < \infty$ and $(a_k) \subset \B^n$ be a sequence such that $a_k \to \omega \in \Sp^n$. Then there exists a subsequence $(b_k)$ of $(a_k)$ such that the mapping $V\colon \ell^p \to H^p$ defined as  $$ V(\alpha) = \sum_{k = 1}^\infty \alpha_k f_{b_k},$$ where $\alpha = (\alpha_k) \in \ell^p$, is bounded.
\end{prop}
\begin{proof}
One just needs to follow the proof given in the one-dimensional case given in \cite[Proposition 3.2]{M1} using our Lemma \ref{testfunctions} as a replacement of Lemma 3.1 of \cite{M1}. We left the details to the interested reader.
\end{proof}

\begin{prop}
\label{normlimit}
Let $b \in \BMOA(\Bn) \setminus \VMOA(\Bn)$ and $1 \le p < \infty$. Then $$c:= \limsup_{|a| \to 1}\|J_b f_a\|_{H^p} > 0.$$ In particular, there exists a sequence $(a_k) \subset \B^n$ such that $a_k \to \omega \in \Sp^n$ and $$\lim_{k \to \infty}\|J_b f_{a_k}\|_{H^p} = c.$$
\end{prop}
\begin{proof}
We may assume $b(0) = 0$. We consider first the case $p  > 2$ and utilize the representation \cite[Chapter 5]{ZhuBn} $$\|b\|^2_{\BMOA} \asymp \sup_{\|f\|_{H^p}\asymp 1}\int_{\B^n}|f(z)|^p|Rb(z)|^2(1-|z|^2) dv(z),$$ where $dv(z)$ is the normalized volume measure on $\Bn$. Now for $f \in H^p,$ by the estimates obtained in pages 144-145 in \cite{Pau2016}, we have
\[
\int_{\B^n}|f(z)|^p|Rb(z)|^2(1-|z|^2)\, dv(z)\le C \|f\|_{H^p}^{p-2} \cdot \|J_bf\|_{H^p}^2,\qquad 2<p<\infty;
\]
and
\[
\int_{\B^n} |f(z)|^p |Rb(z)|^2 (1-|z|^2)\, dv(z)\le C \, \|b\|^{2-p}_{\BMOA} \cdot \|J_b f\|_{H^p}^p,\qquad 1\le p \le 2.
\]
By replacing $f$ with $f_a$ (note that $\|f_ a\|_{H^p}\asymp 1$), we have
\[
\|J_bf_a\|_{H^p}^2 \ge C \int_{\B^n}|f_a(z)|^p|Rb(z)|^2(1-|z|^2)\, dv(z),\qquad  2 < p < \infty;
\]
and
\[
\|b\|^{2-p}_{\BMOA} \cdot \|J_b f_a\|^p_{H^p} \ge C \int_{\B^n}|f_a(z)|^p|Rb(z)|^2(1-|z|^2)\, dv(z), \qquad 1 \le p \le 2.
 \]
It is well known \cite[Chapter 5]{ZhuBn} that a holomorphic function $g$ belongs to $VMOA(\Bn)$ if and only if
\[
\lim_{|a| \to 1^{-}} \int_{\Bn} \frac{(1-|a|^2)}{|1-\langle z,a \rangle |^{n+1}} \,|Rg(z)|^2(1-|z|^2)\, dv(z)=0.
\]
 Since $b \in \BMOA(\Bn) \setminus \VMOA(\Bn),$ it holds that
 \[
 \limsup_{|a| \to 1^{-}}\int_{\B^n}|f_a(z)|^p|Rb(z)|^2(1-|z|^2)\, dv(z) > 0
 \]
 and hence $\limsup_{|a| \to 1}\|J_b f_a\|_{H^p} >0$  for all $1 \le p < \infty.$
\end{proof}

As a final step towards the proof of Theorem \ref{T_gstrictsing}, we construct an isomorphism from $\ell^p$ into $H^p$ using a non-compact $J_b$ and test functions.
\begin{prop}
\label{isomorphism}
Let $b \in BMOA(\Bn) \setminus VMOA(\Bn), \, 1 \le p < \infty,$ and $(a_k) \subset \B^n$ be the sequence from Proposition \ref{normlimit}. Then there exists a subsequence $(b_k)$ of $(a_k)$ such that the mapping $U\colon \ell^p \to H^p,\, U(\alpha)=\sum_{k = 1}^\infty \alpha_k J_b f_{b_k},$ where $\alpha = (\alpha_k) \in \ell^p$, is an isomorphism onto its range.
\end{prop}
\begin{proof}
With the use of Propositions \ref{bddmap}, \ref{normlimit} and Lemma \ref{testfunctions}, we just need to follow the argument given in the one-dimensional case, see \cite[Prop. 3.5]{M1}. We omit the details.
\end{proof}

\begin{proof}[\textbf{Proof of Theorem \ref{T_gstrictsing}}]
By Proposition \ref{bddmap} and Proposition \ref{isomorphism}, we can choose a sequence $(b_k) \subset \B^n$ that  induces a bounded operator $V \colon \ell^p \to H^p,$ given by   $$V(\alpha) = \sum_{k = 1}^\infty \alpha_k f_{b_k},$$ where $\alpha = (\alpha_k) \in \ell^p,$ and an isomorphism $U\colon \ell^p \to H^p,\, U = J_bV$
onto its range.

Define $M = \overline{\textup{span}\{f_{b_k}\}},$ where the closure is taken in $H^p.$
Since $U$ is bounded below, we have that the restriction $J_b|_M$ is also bounded  below. Thus $J_b|_M\colon M \to J_b(M)$ is an isomorphism and consequently $M$ is isomorphic to $\ell^p.$ In particular, the operator $J_b$ is not $\ell^p$-singular.
\end{proof}

\section{$\ell^2$-singularity of $J_b$}\label{s3}

In this section, we show that if  $J_b\colon H^p \to H^p$ is bounded below on a closed infinite-dimensional subspace $M$ of $H^p$, then there exists a subspace $N \subset M$ isomorphic to $\ell^p.$ In particular, this implies that $J_b\colon H^p \to H^p$ cannot fix an isomorphic copy of $\ell^2$ whenever $p \ne 2.$

For $\zeta \in \Sn$, the admissible approach region $\Gamma(\zeta)$ is defined as
\begin{displaymath}
\Gamma(\zeta)=\left \{z\in \Bn: |1-\langle z,\zeta\rangle |< 1-|z|^2 \right \}.
\end{displaymath}
If $I(z):=\{\zeta \in \Sn: z\in \Gamma(\zeta)\}$, then $\sigma(I(z))\asymp (1-|z|^2)^{n}$, and it follows from Fubini's theorem that, for a finite positive measure $\nu$, and a positive function $\varphi$, one has
\begin{equation}\label{EqG}
\int_{\Bn} \varphi(z)\,d\nu(z)\asymp \int_{\Sn} \left (\int_{\Gamma(\zeta)} \varphi(z) \frac{d\nu(z)}{(1-|z|^2)^{n}} \right )d\sigma(\zeta).
\end{equation}

For convenience, we define the measure $\mu_b$ by $$d\mu_b(z)=|Rb(z)|^2(1-|z|^2)dv(z),$$ where $dv$ is the normalized volume measure on $\Bn$,
and set $dV_\alpha(z)=(1-|z|^2)^\alpha dv(z)$. It is well known that a holomorphic function $b$ on $\Bn$ belongs to $BMOA(\Bn)$ if and only if $\mu_ b$ is a Carleson measure; and $b\in VMOA(\Bn)$ if and only if $\mu_ b$ is a vanishing Carleson measure. We recall that a positive Borel measure $\mu$ on $\Bn$ is a Carleson measure if there exists a constant $C>0$ such that
\[\mu \big (B_{\zeta}(\delta)\big )\le C \delta \, ^ n\]
for all $\zeta \in \Sn$ and $\delta>0$. Here $B_\zeta(\delta) = \{z \in \mathbb{B}^n: |1- \langle z, \zeta \rangle| < \delta\}$. Also, $\mu$ is a vanishing Carleson measure if
\[
\mu(B_\zeta(\delta))=o(\delta^n),\qquad \textrm {as }\delta \to 0.
\]
It is also well known that, if $\mu$ is a vanishing Carleson measure then
\[
\int_{\Bn} |f_ k|^p \,d\mu \rightarrow 0
\]
for any bounded sequence of functions $\{f_k\} \subset H^p$ converging to zero uniformly on compact subsets of $\Bn$, $1\le p<\infty$.
Next, we establish some preliminary results en route to the proof of Theorem \ref{mt2}.

\begin{lemma}
\label{le_vcarleson}
Let $\varepsilon  >0$ and $b \in H^2.$ Then there exists a compact set $K_\varepsilon \subset \mathbb{S}^n$ with $\sigma(\mathbb{S}^n\setminus K_\varepsilon) < \varepsilon$ such that $\sup_{\zeta \in K_\varepsilon}\mu_b(B_\zeta(\delta))=o(\delta^n)$ as $\delta \to 0,$
and $\mu_{b,\varepsilon} = \chi_{\Omega_\varepsilon}|Rb|^2\, dV_1$ is a vanishing Carleson measure, where $\Omega_\varepsilon = \bigcup_{\zeta \in K_\varepsilon} \Gamma(\zeta)$.
\end{lemma}
\begin{proof}
For each $k \ge 1$, let $\nu_k$ be the projection to  $\mathbb{S}^n$ of the measure $\mu_b$
restricted to the annulus $S_k = \{z \in \mathbb{B}^n : 1-1/k
< |z| < 1\}.$ That is, $\nu_k$ is determined by the
condition
\[
 \nu_k(I(\zeta, \delta))= \mu_b\big (\{z \in S_k: |1-\langle z, \zeta \rangle|< \delta\}\big),
\]
  where $I(\zeta,\delta) = \{\xi\in\mathbb S^n : |1-\langle \xi, \zeta \rangle|<\delta\}.$ Consider the Hardy-Littlewood maximal function of $\nu_k:$ $$\nu_k^*(\zeta) = \sup_{\delta > 0}\frac{\nu_k(I(\zeta, \delta))}{\sigma(I(\zeta, \delta))} \asymp \sup_{\delta > 0}\frac{\nu_k(I(\zeta, \delta))}{\delta^n}.$$  The maximal function theorem \cite[Chapter 4]{ZhuBn} implies that it satisfies
\begin{equation}
\label{eq: estimate1}
\sigma(\{\zeta \in \mathbb{S}^n: \nu_k^*(\zeta)  >\lambda\}) \lesssim \frac{\nu_k(\mathbb{S}^n)}{\lambda}
\end{equation}
 for all $\lambda > 0$. Since $\mu_b$ is a finite measure (by the Littlewood-Paley identity we have $\mu_ b(\Bn)\asymp \|b\|_{H^2}^2$), by the absolute continuity of the integral, we deduce that $\nu_k(\mathbb{S}^n) = \mu_b(S_k) \to 0$ as
$k \to \infty$. Hence $\nu_k^* \to 0$  almost everywhere on $\mathbb{S}^n$ as $k \to \infty$ by \eqref{eq: estimate1}.
Egorov's theorem now implies that there is a set $F \subset \mathbb{S}^n$ with $\sigma(\mathbb{S}^n \setminus F) < \varepsilon/2$ such that $\nu_k^* \to 0$ uniformly in $F$ as $k \to \infty$. Now for every $k \ge 1$ and $\zeta \in F$, we have
\[
\sup_{0 < \delta < 1/k}\frac{\mu_b(B_\zeta(\delta))}{\delta^n} \le \sup_{\delta  > 0}\frac{\nu_k(I(\zeta,\delta))}{\delta^n} \asymp \nu_k^*(\zeta),
\]
where the first inequality follows from the fact that $B_\zeta(\delta) \subset S_k$ for all $0 < \delta < 1/k.$ Thereby we deduce that $\sup_{\zeta \in F} \mu_b(B_\zeta(\delta)) = o(\delta^n$) as $\delta \to 0$. Hence, if we pick a compact subset $K_\varepsilon \subset F$ with $\sigma(F \setminus K_\varepsilon) < \varepsilon/2,$ we get
\begin{equation}\label{CMKe}
\sup_{\zeta \in K_\varepsilon}\mu_b(B_\zeta(\delta))=o(\delta^n),\qquad \textrm{as }\delta \to 0.
\end{equation}
In order to see that $\mu_{b,\varepsilon}$ is a vanishing Carleson measure, we need to prove that $\mu_{b,\varepsilon}(B_{\zeta}(\delta))=o(\delta^n)$ as $\delta \to 0$ for every $\zeta \in \Sn$.
From \eqref{CMKe}, we obtain that $\sup_{\zeta \in K_\varepsilon} \mu_{b, \varepsilon}(B_\zeta(\delta)) = o(\delta^n$) as $\delta \to 0$. Let $\zeta \in \mathbb{S}^n \setminus K_\varepsilon$ and $\delta  > 0$ small enough. If $B_\zeta(\delta) \cap \Omega_\varepsilon \neq  \emptyset$, then there is a point $w\in B_{\zeta}(\delta)\cap \Gamma(x)$ for some $x\in K_{\varepsilon}$. Using that $d(z,w)=|1-\langle z,w \rangle |^{1/2}$ satisfies the triangle inequality \cite[Proposition 5.1.2]{Rud}, for $z\in B_{\zeta}(\delta)$, we have
\[
\begin{split}
|1-\langle z,x \rangle |^{1/2}&\le |1-\langle z,w \rangle |^{1/2}+|1-\langle w,x \rangle |^{1/2}
\\
& \le |1-\langle z,\zeta \rangle |^{1/2}+|1-\langle \zeta,w \rangle |^{1/2}+|1-\langle w,x \rangle |^{1/2}
\\
&\le 2\delta^{1/2}+(1-|w|^2)^{1/2}
\\
&\le 2\delta^{1/2}+\sqrt{2}\,|1-\langle \zeta,w \rangle |^{1/2}<4 \delta^{1/2}.
\end{split}
\]
Hence $B_{\zeta}(\delta)\subset B_ x \big (16\delta \big)$, and therefore
\[
\mu _{b,\varepsilon}\big (B_{\zeta}(\delta) \big )\le \mu_{b,\varepsilon} \big (B_ x(16\delta)\big )=o(\delta^n),\qquad \textrm{as }\delta \to 0.
\]
It now follows that the measure  $\mu_{b, \varepsilon}$ is a vanishing Carleson measure.
\end{proof}

For $1\le p<\infty$ and a sequence of functions  $\{f_ k\}\subset H^p$, given a subset $A$ of $\Sn$, we consider the quantities
\begin{align*}
A(j,k)&=\int_{A}\left(\int_{\Gamma(\xi)}|J_b f_j|^{p-2}|f_k|^2 |Rb|^2 dV_{1-n}\right) d\sigma(\xi); \\
A(\infty,k) &= \sup_{\|f\|_{H^p}=1} \int_{A}\left(\int_{\Gamma(\xi)}|J_b f|^{p-2}|f_k|^2 |Rb|^2 dV_{1-n}\right) d\sigma(\xi);
\\
A(k, \infty) &= \sup_{\|f\|_{H^p}=1} \int_{A}\left(\int_{\Gamma(\xi)}|J_b f_k|^{p-2}|f|^2 |Rb|^2 dV_{1-n}\right) d\sigma(\xi).
\end{align*}
\begin{lemma}
\label{le_functionals}
Let $b \in BMOA(\Bn)$,
$0 < \delta < 1$ and
$\{f_k\}$ be a normalized sequence in $H^p$, which converges to zero uniformly on compact subsets of $\Bn.$ If $p>2$,
there exists a subsequence denoted still by $\{f_k\}$, a decreasing sequence $\varepsilon_m > 0, \, \varepsilon_m \to 0,$ and compact sets $K_m\subset \Sn$ satisfying $K_m\subset K_{m+1}$ and $\sigma(E_m) < \varepsilon_m$, where $E_m = \mathbb{S}^n \setminus K_m$ such that
\begin{align*}
&E_m(j,k)\lesssim \delta^2 4^{-j-k-m}\ \ \ for\ j,k<m;\\
&E_m(\infty,k)\lesssim \delta^2 4^{-k-2m}\ \ \ for\ k<m;\\
&E_m(k,\infty)\lesssim \delta^2 4^{-k-2m}\ \ \ for\ k<m;\\
&K_m(k,m)\lesssim \delta^2 4^{-k-2m}\ \ \ for\ k\leq m;\\
&K_m(m,k)\lesssim \delta^2 4^{-k-2m}\ \ \ for\ k\leq m,
\end{align*}
for all $m\geq1$.
In particular, by defining $\tilde{E}_m = E_m \setminus E_{m+1}$,
we have that $\tilde{E}_m(j,k)\lesssim \delta^2 4^{-j-k-m}$for $k \ne m$ or $j \ne m$.
\end{lemma}

\begin{proof}Since $b \in BMOA(\Bn)$, the operator $J_ b$ is bounded on $H^p$. Also, Lemma 4.1 implies that for any $\varepsilon>0$, there exists a compact set $K_\varepsilon \subset \mathbb{S}^n$ with $\sigma(E_\varepsilon) < \varepsilon$ where $E_\varepsilon = \mathbb{S}^n \setminus K_\varepsilon$ such that
\begin{equation}
\label{eq: vcarleson}
\mu_{b,\varepsilon} = \chi_{\Omega_\varepsilon}|Rb|^2\, dV_1
\end{equation} is a vanishing Carleson measure. Note that, by \eqref{EqG} and H\"{o}lder's inequality with exponent $p/2>1$,
\[
\begin{split}
\int_{\Sn} & \left(\int_{\Gamma(\xi)}|J_b f_k|^{p-2} |f_m|^2 |Rb|^2 dV_{1-n}\right)d\sigma(\xi)
 \asymp  \int_{\Bn}|J_b f_k|^{p-2}|f_m|^2 d\mu_b
\\
&\leq \left[ \int_{\Bn} |J_bf_k|^p d\mu_b\right]^{\frac{p-2}{p}}\left[\int_{\Bn} |f_m|^p d\mu_b\right]^{\frac{2}{p}}
\lesssim \|J_ b f_ k\|_{H^p}^{p-2}\cdot \|f_ m\|_{H^p}^2<\infty.
\end{split}
\]
The last estimate is due to Carleson-H\"{o}rmander theorem on Carleson measures.
Therefore, by absolute continuity, for all fixed $(k,m) \in \mathbb{N}^2$, one has
\begin{equation}\label{eq: 3}
\lim_{\varepsilon\to 0} \int_{E_\varepsilon} \left(\int_{\Gamma(\xi)}|J_b f_k|^{p-2}|f_m|^2 |Rb|^2 dV_{1-n}\right)d\sigma(\xi)=0.
\end{equation}
As a simple application of \eqref{EqG}, for any positive measurable function $\varphi$, we have
\begin{equation}\label{eq: 4}
\int_{K_\varepsilon}\left(\int_{\Gamma(\xi)}\frac{\varphi(z)dv(z)}{(1-|z|^2)^n}\right) d\sigma(\xi)\lesssim \int_{\Omega_\varepsilon} \varphi(z)dv(z).
\end{equation}
By repeating the calculation above and using this formula, we obtain
\begin{equation}\label{eq: 5}
\lim_{(k,m)\to \infty}\int_{K_\varepsilon}\left(\int_{\Gamma(\xi)}|J_b f_k|^{p-2}|f_m|^2 |Rb|^2 dV_{1-n}\right)d\sigma(\xi)=0,
\end{equation}
which follows from the vanishing Carleson measure condition \eqref{eq: vcarleson}, where $(k,m)\to \infty$ means that $k+m\to \infty$.\\

For $f\in H^p$, the admissible maximal function $f^*(\zeta)=\sup_{z\in \Gamma(\zeta)}|f(z)|$ is bounded on $L^p(\Sn)$, that is, $\|f^*\|_{L^p(\Sn)}\lesssim \|f\|_{H^p}$ (see \cite[Chapter 4]{ZhuBn}). Assume now that $f,g \in H^p$ are unit vectors, then
\[
\begin{split}
\int_{E_\varepsilon} &\left(\int_{\Gamma(\xi)} |J_b f|^{p-2}|g|^2 |Rb|^2 dV_{1-n}\right) d\sigma(\xi)\\
&\lesssim  \int_{E_\varepsilon} (J_b f)^*(\xi)^{p-2} \left(\int_{\Gamma(\xi)} |g|^2 |Rb|^2 dV_{1-n}\right) d\sigma(\xi)\\
&\lesssim  \left(\int_{E_\varepsilon} (J_b f)^*(\xi)^p d\sigma(\xi)\right)^{\frac{p-2}{p}}\left(\int_{E_\varepsilon} \left(\int_{\Gamma(\xi)} |g|^2 |Rb|^2 dV_{1-n}\right)^{\frac{p}{2}} d\sigma(\xi)\right)^{\frac{2}{p}}.
\end{split}
\]
Observe that both factors go to zero as $\sigma(E_\varepsilon)\to 0$ due to the absolute continuity of the measure, as the boundedness of $J_ b$ gives $\|(J_ b f)^*\|_{L^p(\Sn)}\lesssim \|J_ b f \|_{H^p}\le \|J_ b\|$ and, by the version of Calder\'{o}n area theorem for the unit ball \cite{AB,Pau2016}, we have
\[
\int_{\Sn} \left(\int_{\Gamma(\xi)} |g|^2 |Rb|^2 dV_{1-n}\right)^{\frac{p}{2}} d\sigma(\xi)\asymp \|J_ b g \|^p_{H^p}\le \|J_ b\|^p.
\]
Hence, as $\varepsilon \rightarrow 0$, we have
\begin{equation}\label{Eq-k25}
E_{\varepsilon}(\infty,k) \lesssim  \|J_ b  \|^{p-2} \left(\int_{E_\varepsilon} \left(\int_{\Gamma(\xi)} |f_ k|^2 |Rb|^2 dV_{1-n}\right)^{\frac{p}{2}} d\sigma(\xi)\right)^{\frac{2}{p}} \longrightarrow 0.
\end{equation}
and
\begin{equation}\label{Eq-k50}
E_{\varepsilon}(k,\infty) \lesssim \left(\int_{E_\varepsilon} (J_b f_ k)^*(\xi)^p d\sigma(\xi)\right)^{\frac{p-2}{p}}\,  \|J_ b  \|^{2}  \longrightarrow 0.
\end{equation}
We will choose a subsequence of $(f_k)$, which we will also denote as $(f_k)$, and $\varepsilon_1 > \varepsilon_2 \dots >0$ in the following way.
Assume that functions $f_1,\ldots, f_{m-1}$, numbers $\varepsilon_1 > \ldots > \varepsilon_{m-1} > 0$ and compact sets $K_1 \subset \ldots \subset K_{m-1} $ are chosen for some $m \ge 2$.
Then \eqref{eq: 3} together with \eqref{Eq-k25} and \eqref{Eq-k50} yields that there exists $\varepsilon_m < \varepsilon_{m-1}$ with $K_{m-1}\subset K_{m}$ and  $\sigma(E_m) < \varepsilon_m$ such that
$E_m(j,k) \lesssim \delta^2 4^{-j-k-m}$ for every $j,k < m$, and $E_m(\infty,k)\lesssim \delta^2 4^{-k-2m}$, $E_m(k,\infty)\lesssim \delta^2 4^{-k-2m}$ for $k<m$.
After that we can use \eqref{eq: 5} to find $f_m$ such that
$K_m(k,m) \lesssim \delta^2 4^{-k-2m}$
and $K_m(m,k) \lesssim \delta^2 4^{-k-2m}$ for $k\leq m$.
Hence for $j,k<m$, we have
\[
\begin{split}
\tilde{E}_m(j,k)&=\int_{\tilde{E}_m}\left(\int_{\Gamma(\xi)}|J_b f_j|^{p-2}|f_k|^2 |Rb|^2 dV_{1-n}\right) d\sigma(\xi) \le
E_m(j,k) 
\lesssim \delta^2 4^{-j-k-m}.
\end{split}
\]
For $k<m$, $\tilde{E}_m(m,k) \le E_m(\infty,k) \lesssim \delta^2 4^{-k-2m}$ and $\tilde{E}_m(k,m) \le E_m(k,\infty) \lesssim \delta^2 4^{-k-2m}$.

Also for the case $\max\{j,k\}>m$, if $j=\max\{j,k\}>m$, we may use $\tilde{E}_m= E_m\setminus E_{m+1}\subset\mathbb{S}^n\setminus E_{m+1}=K_{m+1}\subset K_j$ to obtain
\begin{align*}
 \tilde{E}_m(j,k)\leq K_j(j,k)\lesssim \delta^2 4^{-2j-k}\leq \delta^2 4^{-j-k-m}.
\end{align*}
Analogously, we have $\tilde{E}_m(j,k)\leq K_k(j,k)\lesssim \delta^2 4^{-j-k-m}$ for $k=\max\{j,k\}>m$.
Thus, $\tilde{E}_m(j,k)\lesssim \delta^2 4^{-j-k-m}$ for $k\neq m$ or $j\neq m$.
\end{proof}

Next, we establish an analogous version of Lemma \ref{le_functionals} in the case $1 \le p \le 2$.

\begin{lemma}
\label{le_functionals-bis}
 Let $1 \le p \le 2$,  $b \in BMOA(\Bn)$,
$0 < \delta < 1$, and
$\{f_k\}$ be a normalized sequence in $H^p$, which converges to zero uniformly on compact subsets of $\Bn.$ Then there exists a subsequence denoted still by $\{f_k\}$, a decreasing sequence $\varepsilon_m > 0, \, \varepsilon_m \to 0,$ and compact sets $K_m\subset \Sn$ satisfying $K_m\subset K_{m+1}$ and $\sigma(E_m) < \varepsilon_m$, where $E_m = \mathbb{S}^n \setminus K_m$ such that
$$L_m(m) = \int_{K_m}\left(\int_{\Gamma(\xi)}|f_m|^2 |Rb|^2 dV_{1-n}\right)^{\frac{p}{2}} d\sigma(\xi)$$ and $$F_m(k) =  \int_{E_m}\left(\int_{\Gamma(\xi)}|f_k|^2 |Rb|^2 dV_{1-n}\right)^{\frac{p}{2}} d\sigma(\xi)$$ satisfy $F_m(k) \lesssim \delta 4^{-k-m}$ for $k < m$ and $L_m(m) \lesssim \delta 4^{-2m}$ for all $m\geq1$. In particular, $$\tilde{F}_m(k)=\int_{\tilde{E}_m}\left(\int_{\Gamma(\xi)}|f_k|^2 |Rb|^2 dV_{1-n}\right)^{\frac{p}{2}} d\sigma(\xi) \lesssim \delta 4^{-k-m}$$ for $k \ne m$, where $\tilde{E}_ m=E_ m \setminus E_{m+1}$. Also, we have $\tilde{F}_m(m) \lesssim 1$.
\end{lemma}

\begin{proof}
As before, the operator $J_ b$ is bounded on $H^p$ and for any $\varepsilon>0$, there exists a compact set $K_\varepsilon \subset \mathbb{S}^n$ with $\sigma(E_\varepsilon) < \varepsilon$  such that $
\mu_{b,\varepsilon}$ is a vanishing Carleson measure. By the version of Calder\'{o}n's area theorem for the unit ball \cite{AB,Pau2016},
\begin{align*}
\int_{\mathbb{S}^n}\left(\int_{\Gamma(\xi)}|f_k|^2 |Rb|^2 dV_{1-n}\right)^{\frac{p}{2}}d\sigma(\xi)\asymp\|J_bf_k\|_{H^p}^{p}\lesssim1.
\end{align*}
In particular, $\tilde{F}_m(k)\lesssim 1$. Therefore, by absolute continuity, for all $k$, one has
\begin{equation}\label{eq: 6}
\lim_{\varepsilon\to 0} \int_{E_\varepsilon}\left(\int_{\Gamma(\xi)}|f_k|^2 |Rb|^2 dV_{1-n}\right)^{\frac{p}{2}}d\sigma(\xi)=0.
\end{equation}
Now, using H\"{o}lder's inequality, the $L^p$-boundedness of the admissible maximal function and \eqref{eq: 4}, we get
\begin{align*}
\int_{K_\varepsilon}\left(\int_{\Gamma(\xi)}|f_m|^2 |Rb|^2 dV_{1-n}\right)^{\frac{p}{2}} & d\sigma(\xi)
\\&\le \int_{K_\varepsilon}|f^*_m(\xi)|^{\frac{(2-p)p}{2}}\left(\int_{\Gamma(\xi)}|f_m|^p |Rb|^2 dV_{1-n}\right)^{\frac{p}{2}}d\sigma(\xi)\\
&\le \|f^*_m\|_{L^p(\mathbb{S}^n)}^{\frac{(2-p)p}{2}}\left(\int_{K_\varepsilon}\int_{\Gamma(\xi)}|f_m|^p |Rb|^2 dV_{1-n}d\sigma(\xi)\right)^{\frac{p}{2}}\\
&\lesssim \left(\int_{\Omega_\varepsilon} |f_m|^p d\mu_b \right)^{\frac{p}{2}}.
\end{align*}
As $\mu_{b,\varepsilon}$ is a vanishing Carleson measure,  we obtain
\begin{equation}
\label{eq: vcarleson2}
\lim_{m \to \infty}\int_{K_\varepsilon}\left(\int_{\Gamma(\xi)}|f_m|^2 |Rb|^2 dV_{1-n}\right)^{\frac{p}{2}}d\sigma(\xi)=0.
\end{equation}

As in the proof of Lemma \ref{le_functionals}, we may use \eqref{eq: 6} and \eqref{eq: vcarleson2} inductively to find a subsequence denoted still by $\{f_k\}$, a decreasing sequence $\varepsilon_m > 0, \, \varepsilon_m \to 0,$ and compact sets $K_m$ satisfying $K_m\subset K_{m+1}$ and $\sigma(E_m) < \varepsilon_m$ where $E_m = \mathbb{S}^n \setminus K_m$ such that $F_m(k) \lesssim \delta 4^{-k-m}$ for $k < m$ and $L_m(m) \lesssim \delta 4^{-2m}$ for all $m\geq1$.

More precisely, we have
\begin{align*}
\tilde{F}_m(k) = \int_{\tilde{E}_m}\left(\int_{\Gamma(\xi)} |f_k|^2 |Rb|^2 dV_{1-n}\right)^{\frac{p}{2}} d\sigma(\xi) \le F_m(k) 
\lesssim \delta 4^{-k-m}
\end{align*}
for $k < m$. Moreover, it holds that $\tilde{E}_m \subset K_{m+1} \subset K_k$ for $k > m$. Hence
\[
\begin{split}
\tilde{F}_m(k) \le
L_k(k)\lesssim \delta 4^{-2k}
< \delta 4^{-k-m}
\end{split}
\]
for $k > m$.
Therefore, it always holds that $\tilde{F}_m(k)\lesssim \delta 4^{-k-m}$ for $k\neq m$.

\end{proof}

 We now construct a bounded linear operator acting on $H^p$ in terms of the operator $J_b$ and a normalized sequence of functions converging uniformly to zero on compact subsets of $\B_n$.

\begin{lemma}
\label{bddU}
Let $1 \le p < \infty$, $b \in BMOA(\Bn)$,
and $(f_k) \subset H^p$ be such that $\|f_k\|_{H^p} = 1$ for all $k$ and $(f_k)$ converges to zero uniformly on compact subsets of $\Bn.$ Then there exists a subsequence $(f_{n_k})$ of $(f_k)$ such that the linear mapping $U\colon \ell^p \to H^p,$ defined as $$U(\alpha) = \sum_{k = 1}^\infty \alpha_k J_bf_{n_k},$$ where $\alpha = (\alpha_k) \in \ell^p$, is bounded.
\end{lemma}
\begin{proof}

We divide our proof into three cases depending on the value of $p$, namely cases $1 \le p \le 2,$ $2 < p \le 3$ and $3 < p < \infty$. This results from the use of norms, which are equivalent to the standard $H^p$ norm and the choice of a particular norm depends on the value of $p$. For $0<\delta<1$, we choose a subsequence of $\{f_k\}$ denoted still by $\{f_k\}$, a decreasing sequence $\varepsilon_m > 0, \, \varepsilon_m \to 0,$ and compact sets $K_m$ from Lemmas \ref{le_functionals} and \ref{le_functionals-bis}.
Set $\widetilde{E}_0 = \mathbb{S}^n\setminus E_1=K_1$.\\

Let us first look at the case $1 \le p \le 2.$ Set $\alpha_ 0=0$.
By the version of the area theorem for the unit ball \cite{AB,Pau2016}, we have
\begin{align*}
\|U(\alpha)\|_{H^p}^p&\asymp \int_{\mathbb{S}^n} \left( \int_{\Gamma(\zeta)}\left| \sum_{k=1}^\infty \alpha_k f_kRb\right|^2 \, dV_{1-n}\right)^{\frac{p}{2}}\, d\sigma(\zeta)\\
&= \sum_{m=0}^\infty \int_{\tilde{E}_m} \left( \int_{\Gamma(\zeta)}\left| \sum_{k=1}^\infty \alpha_k f_kRb\right|^2 \, dV_{1-n}\right)^{\frac{p}{2}}\, d\sigma(\zeta) \\
&\lesssim\sum_{m=0}^\infty \int_{\tilde{E}_m} \left(\int_{\Gamma(\zeta)} |\alpha_m f_mRb|^2 dV_{1-n}+\int_{\Gamma(\zeta)}\left| \sum_{\substack{k=1\\ k\ne m}}^\infty \alpha_k f_kRb\right|^2 \, dV_{1-n}\right)^{\frac{p}{2}}\, d\sigma(\zeta).
\end{align*}
According to the assumption $1 \le p \le 2$, then $d(F,G)=\|F-G\|_{L^{\frac{p}{2}}}^{\frac{p}{2}}$ is a metric and hence
\[
\begin{split}
\|U(\alpha)\|_{H^p}^p&\lesssim\sum_{m=1}^\infty \int_{\tilde{E}_m} \left(\int_{\Gamma(\zeta)} |\alpha_m f_mRb|^2 dV_{1-n}\right)^{\frac{p}{2}}\, d\sigma(\zeta)
\\
&\ \ \
+\sum_{m=0}^\infty \int_{\tilde{E}_m} \left(\int_{\Gamma(\zeta)}\left| \sum_{\substack{k=1\\ k\ne m}}^\infty \alpha_k f_kRb\right|^2 \, dV_{1-n}\right)^{\frac{p}{2}}\, d\sigma(\zeta).
\end{split}
\]
Using the triangle inequalities in $L^2$ and $L^p$ respectively, one has
\begin{align*}
\sum_{m=0}^\infty &\int_{\tilde{E}_m} \left(\int_{\Gamma(\zeta)}\Big| \sum_{\substack{k=1\\ k\ne m}}^\infty \alpha_k f_kRb\Big|^2 \, dV_{1-n}\right)^{\frac{p}{2}}\, d\sigma(\zeta) \\
&\leq\sum_{m=0}^\infty \int_{\tilde{E}_m}\left[ \sum_{\substack{k=1\\ k\ne m}}^\infty |\alpha_k|\left(\int_{\Gamma(\zeta)}| f_kRb|^2 \, dV_{1-n}\right)^{\frac{1}{2}}\, \right]^p d\sigma(\zeta) \\
&\leq \sum_{m=0}^\infty \left(\sum_{\substack{k=1\\ k\ne m}}^\infty |\alpha_k|\left[ \int_{\tilde{E}_m}\left(\int_{\Gamma(\zeta)}| f_kRb|^2 \, dV_{1-n}\right)^{\frac{p}{2}}\, d\sigma(\zeta) \right]^{\frac{1}{p}}\right)^p
\\
& =\sum_{m=0}^\infty \Big(\sum_{\substack{k=1\\ k\ne m}}^\infty |\alpha_k|\,\tilde{F}_m(k)^{1/p}\Big)^{p}.
\end{align*}
Since $K_1\subset K_k$ for any $k\geq 1$ and $\tilde{F}_ 0=K_ 1$, we have $\tilde{F}_0(k)\le L_k(k)\lesssim \delta 4^{-2k}\le \delta 4^{-k}$. Therefore, together with the estimates $\tilde{F}_m(m)\lesssim1$ and $\tilde{F}_m(k)\lesssim\delta 4^{-k-m}$ for $k\ne m$, we have
\[
\begin{split}
\|U(\alpha)\|_{H^p}^p &\lesssim \sum_{m=1}^\infty |\alpha_{m}|^p\tilde{F}_m(m)+ \sum_{m=0}^\infty \Big(\sum_{\substack{k=1\\ k\ne m}}^\infty |\alpha_k|\tilde{F}_m(k)^{1/p}\Big)^{p} \\
&\lesssim \|\alpha\|^p_{\ell^p} \left(1+\sum_{m=0}^\infty \left(\sum_{k=1}^\infty (4^{-k-m})^{\frac{1}{p}}\right)^{p}\right)\lesssim \|\alpha\|^p_{\ell^p}.
\end{split}
\]
\mbox{}
\\

Let us then consider the case $p > 2$. Since $K_1\subset K_k$ for any $k\geq 1$, we have $\tilde{E}_0(j,k)\le K_1(j,k)\lesssim \delta^2 4^{-j-k}$.
Our starting point will be the estimate (a consequence of the Hardy-Stein estimates together with \eqref{EqG})
\begin{align*}
\|U(\alpha)\|_{H^p}^p &\asymp \int_{\mathbb{S}_n}\left(\int_{\Gamma(\xi)}\left|\sum_{j=1}^\infty \alpha_j J_b f_j\right|^{p-2}\left|\sum_{k=1}^\infty \alpha_k f_k\right|^2 |Rb|^2 dV_{1-n}\right) d\sigma(\xi)\\
&= \sum_{m=0}^\infty \int_{\tilde{E}_m}\int_{\Gamma(\xi)}\left|\sum_{j=1}^\infty \alpha_j J_b f_j\right|^{p-2}d\mu \,d\sigma(\xi),
\end{align*}
where $d\mu = \left|\sum\limits_{k=1}^\infty \alpha_k f_k\right|^2 |Rb|^2 dV_{1-n} $.

Now assume first that $2 < p \le 3.$ Then
\[
\left|\sum_{j=1}^\infty \alpha_j J_b f_j\right|^{p-2} \le \sum_{j=1}^\infty |\alpha_j|^{p-2} |J_b f_j|^{p-2}.
\]
Hence
\begin{equation}\label{Eq42}
\begin{split}
\|U(\alpha)\|_{H^p}^p &\asymp \sum_{m=0}^\infty \int_{\tilde{E}_m}\int_{\Gamma(\xi)}\left|\sum_{j=1}^\infty \alpha_j J_b f_j\right|^{p-2}d\mu d\sigma(\xi)\\
&\le \sum_{m=1}^\infty |\alpha_m|^{p-2}I(m,m)+\sum_{m=0}^\infty \sum_{\substack{j = 1 \\ j \ne m}}^\infty |\alpha_j|^{p-2}I(m,j),
\end{split}
\end{equation}
where
\[
I(m,j):=\int_{\tilde{E}_m}\int_{\Gamma(\xi)}\left| J_bf_j\right|^{p-2}d\mu \, d\sigma(\xi).
\]
Applying the triangle inequality in $L^2$, one has for $m\ge 0$ and $j \ne m$
\begin{align*}
I(m,j) &= \int_{\tilde{E}_m}\int_{\Gamma(\xi)}\left| J_bf_j\right|^{p-2}\left|\sum_{k=1}^\infty \alpha_k f_k\right|^2 |Rb|^2 dV_{1-n}\,d\sigma(\xi)
\\
&\le \left(\sum_{k=1}^\infty |\alpha_k| \left(\int_{\tilde{E}_m}\int_{\Gamma(\xi)}\left| J_bf_j\right|^{p-2}|f_k|^2 |Rb|^2 dV_{1-n}\,d\sigma(\xi)\right)^{1/2} \right)^2
\\
&=\left(\sum_{k=1}^\infty |\alpha_k| \tilde{E}_ m(j,k)^{1/2} \right)^2
\\
&\lesssim \delta^2\|\alpha\|_{\ell^p}^2 \left(\sum_{k=1}^\infty  2^{-j-k-m} \right)^2=\delta^2\|\alpha\|_{\ell^p}^2 4^{-j-m},
\end{align*}
and for $m\ge1$, using again the triangle inequality in $L^2$,
\begin{align*}
I(m,m)
&=  \int_{\tilde{E}_m}\int_{\Gamma(\xi)}\left| J_bf_m\right|^{p-2}\left|\alpha_m f_m +\sum_{\substack{k=1 \\ k \ne m}}^\infty \alpha_k f_k\right|^2 |Rb|^2 dV_{1-n}d\sigma(\xi)
\\
&\le \left (|\alpha_ m| \tilde{E}_ m(m,m)^{1/2}+\sum_{\substack{k=1 \\ k \ne m}}^\infty  |\alpha_ k| \tilde{E}_ m (m,k)^{1/2}\right )^{2}.
\end{align*}
By \eqref{EqG} and the Hardy-Stein estimates, we have $\tilde{E}_ m(m,m)\lesssim \|f_ m\|_{H^p}^p\lesssim 1$. Also, by Lemma \ref{le_functionals}, $\tilde{E}_ m (m,k)\lesssim \delta^2 4^{-k-2m}$ for $k\ne m$. Hence
\[
I(m,m)\lesssim \left( |\alpha_m|+\sum_{\substack{k=1 \\ k \ne m}}^\infty \delta 2^{-k-2m}|\alpha_k|\right)^2 \lesssim \left( |\alpha_m|+\delta \|\alpha\|_{\ell^p}4^{-m}\right)^2.
\]
So, we deduce that
\begin{align}\label{eq: 8}
I(m,j)\lesssim\left\{
\begin{aligned}
&\delta^2\|\alpha\|_{\ell^p}^2 4^{-j-m}\ \ \  \text{for}\  j\ne m,\ m\ge 0\  \text{and}\  j\ge 1;\\
\\
&\left( |\alpha_m|+\delta \|\alpha\|_{\ell^p}4^{-m}\right)^2 \ \ \  \text{for}\  j=m\ge1.
\end{aligned}\right.
\end{align}
Hence, bearing in mind \eqref{Eq42} and $|\alpha_ j|\le \|\alpha\|_{\ell^p}$,  we obtain
\begin{align*}
\|U\alpha\|_{H^p}^p &\lesssim \sum_{m=1}^\infty |\alpha_m|^{p-2}I(m,m)+\sum_{m=0}^\infty \sum_{\substack{j = 1 \\ j \ne m}}^\infty |\alpha_j|^{p-2}I(m,j)
\\
&\lesssim \sum_{m=1}^\infty |\alpha_m|^{p-2}( |\alpha_m|+4^{-m}\|\alpha\|_{\ell^p})^2+ \sum_{m=0}^\infty \sum_{\substack{j = 1 \\ j \ne m}}^\infty |\alpha_j|^{p-2}\|\alpha\|_{\ell^p}^2 4^{-j-m}
\\
&\lesssim \sum_{m=1}^\infty |\alpha_m|^{p-2}( |\alpha_m|^2+4^{-2m}\|\alpha\|^2_{\ell^p})+ \|\alpha\|_{\ell^p}^p
\lesssim \|\alpha\|^p_{\ell^p}.
\end{align*}

Finally, we consider the case  $p > 3$. As before, we have
\[
\|U(\alpha)\|_{H^p}^p \asymp \sum_{m=0}^\infty \int_{\tilde{E}_m}\int_{\Gamma(\xi)}\left|\sum_{j=1}^\infty \alpha_j J_b f_j\right|^{p-2}d\mu d\sigma(\xi).
\]
As $p>3$, we can use the triangle inequality in $L^{p-2}$ in order to obtain
\begin{align*}
\|U(\alpha)\|_{H^p}^p
&\lesssim \sum_{m = 0}^\infty \left(\sum_{j = 1}^\infty |\alpha_j| \left(\int_{\widetilde{E}_m}\int_{\Gamma(\xi)}|J_bf_j|^{p-2}d\mu d\sigma(\xi)\right)^{\frac{1}{p-2}}\right)^{p-2}
\\
&=\sum_{m = 0}^\infty \left(\sum_{j = 1}^\infty |\alpha_j| I(m,j)^{\frac{1}{p-2}}\right)^{p-2}.
\end{align*}
 Observe that the estimates obtained in \eqref{eq: 8} for $I(m,j)$ are valid for all $p>2$. Hence
\begin{align*}
\|U(\alpha)\|_{H^p}^p &\lesssim \sum_{m =1}^\infty|\alpha_m|^{p-2} I(m,m)  +\sum_{m = 0}^\infty\left(\sum_{\substack{j = 1 \\ j \ne m}}^\infty |\alpha_j| I(m,j)^{\frac{1}{p-2}}\right)^{p-2}
\\
&\lesssim \sum_{m =1}^\infty|\alpha_m|^{p-2}\left( |\alpha_m|+\|\alpha\|_{\ell^p}4^{-m}\right)^2  +\|\alpha\|_{\ell^p}^p \sum_{m = 0}^\infty\left(\sum_{\substack{j = 1 \\ j \ne m}}^\infty (4^{-j-m})^{\frac{1}{p-2}}\right)^{p-2}
\\
&\lesssim\|\alpha\|_{\ell^p}^p.
\end{align*}
Hence $U$ is bounded from $\ell^p$ into $H^p$ for $1 \le p <\infty.$
\end{proof}

\begin{lemma}
\label{bddbelU}
Let $1 \le p < \infty$, $b \in BMOA(\Bn)$, 
and $(f_k) \subset H^p$ be such that $\|f_k\|_{H^p} = 1$ for all $k$ and $(f_k)$  converges to zero uniformly on compact subsets of $\Bn$. Assume also that $\inf_k \|J_bf_k\|_p \asymp 1.$ Then there exists a subsequence, still denoted by $(f_k)$, such that the linear mapping $$U\colon \ell^p \to H^p, \, U(\alpha) = \sum_{k = 1}^\infty \alpha_k J_bf_k,$$ where $\alpha= (\alpha_k) \in \ell^p$, is bounded below.
\end{lemma}
\begin{proof} The proof is also divided into three cases depending on the value of $p$, namely cases $1 \le p \le 2,$ $2 < p \le 3$ and $3 < p < \infty$.
For $0<\delta<1$, which will be determined later, we choose a subsequence of $\{f_k\}$ denoted still by $\{f_k\}$, a decreasing sequence $\varepsilon_m > 0, \, \varepsilon_m \to 0,$ and compact sets $K_m$ from Lemmas \ref{le_functionals} and \ref{le_functionals-bis}. We proceed to show that $U$ is bounded below.

We consider first the case $p>2$. As done before, we have
\begin{align*}
\|U(\alpha)\|_{H^p}^p &\asymp \int_{\mathbb{S}_n}\left(\int_{\Gamma(\xi)}\left|\sum_{j=1}^\infty \alpha_j J_b f_j\right|^{p-2}\left|\sum_{k=1}^\infty \alpha_k f_k\right|^2 |Rb|^2 dV_{1-n}\right) d\sigma(\xi)\\
&\geq \sum_{m=1}^\infty \int_{\tilde{E}_m}\int_{\Gamma(\xi)}\left|\sum_{j=1}^\infty \alpha_j J_b f_j\right|^{p-2}d\mu d\sigma(\xi),
\end{align*}
where $d\mu = \left|\sum\limits_{k=1}^\infty \alpha_k f_k\right|^2 |Rb|^2 dV_{1-n} $.

For the case $p > 3$, using  the standard estimate $(a+b)^q \le 2^{q-1}(a^q+b^q),$ where  $a, b \ge 0$ and $q \ge 1 $, we obtain
\begin{align*}
\|U(\alpha)\|_{H^p}^p &\gtrsim \sum_{m=1}^\infty \int_{\tilde{E}_m}\int_{\Gamma(\xi)}\Big| \alpha_m J_bf_m+\sum_{\substack{j = 1 \\ j \ne m}}^\infty \alpha_j J_bf_j\Big|^{p-2}d\mu d\sigma(\xi)
\\
&\geq\sum_{m=1}^\infty \int_{\tilde{E}_m} \int_{\Gamma(\xi)}\left (2^{3-p}\Big|\alpha_m J_bf_m\Big|^{p-2}-\Big|\sum_{\substack{j = 1 \\ j \ne m}}^\infty \alpha_j J_bf_j\Big|^{p-2}\right )d\mu \,d\sigma(\xi)
\\
&\geq  2^{3-p}\sum_{m=1}^\infty |\alpha_m|^{p-2} I(m,m) -\sum_{m =1}^\infty\Big(\sum_{\substack{j = 1 \\ j \ne m}}^\infty |\alpha_j| I(m,j)^{\frac{1}{p-2}}\Big)^{p-2}.
\end{align*}
The last inequality is a consequence of the triangle inequality in $L^{p-2}$ (see the proof of the case $p>3$ in Lemma \ref{bddU}). Also, the quantities $I(m,j)$ are the ones appearing in Lemma \ref{bddU}. By the estimates in \eqref{eq: 8}, we have
\[
\sum_{m =1}^\infty\Big(\sum_{\substack{j = 1 \\ j \ne m}}^\infty |\alpha_j| I(m,j)^{\frac{1}{p-2}}\Big)^{p-2}\le C_ 1 \delta^2 \|\alpha\|^p_{\ell^p}
\]
for some positive constant $C_ 1$. Hence,
\begin{equation}\label{Eq-25}
\|U(\alpha)\|_{H^p}^p \gtrsim 2^{3-p}\sum_{m=1}^\infty |\alpha_m|^{p-2} I(m,m) -C_ 1 \delta^2 \|\alpha\|^p_{\ell^p}.
\end{equation}
Now, the trivial estimate (a consequence of the inequality $|A-B|^2 \le 2 (|A|^2+|B|^2)$)
\[
\big | \sum_{k=1}^{\infty} \alpha _ k \,f_ k \big |^2 \ge \frac{1}{2} \,|\alpha_ m \,f_ m |^2 -\big | \sum_{k\ne m} \alpha_ k \,f_ k |^2,
\]
gives
\[
\begin{split}
I(m,m) &
= \int_{\tilde{E}_m}\int_{\Gamma(\xi)}\left|J_bf_m\right|^{p-2}|\sum_{k = 1}^\infty\alpha_k f_k|^2 |Rb|^2 dV_{1-n}\,d\sigma(\xi)
\\
& \ge \frac{1}{2}\,|\alpha_ m|^2\, \tilde{E}_ m (m,m)-B(m),
\end{split}
\]
with
\[
B(m):=\int_{\tilde{E}_m}\int_{\Gamma(\zeta)}\left|J_bf_m\right|^{p-2}\big |\sum_{\substack{k = 1 \\ k \ne m}}^\infty\alpha_k f_k \big |^2 |Rb|^2 dV_{1-n}\,d\sigma(\zeta).
\]
As  before, we use the triangle inequality in $L^2$ and the estimates $\tilde{E}_m(m,k)\lesssim \delta ^2 4^{-k-2m}$ for $k\ne m$ to obtain
\[
B(m)\leq \Big (\sum\limits_{\substack{k = 1 \\ k \ne m}}^\infty |\alpha_ k|\,\tilde{E}_m(m,k)^{\frac{1}{2}}\Big )^{2}
\leq C_ 2 \delta^2 \, \|\alpha\|^2_{\ell^p}\,4^{-2m}
\]
for some positive constant $C_ 2$. As $\tilde{E}_m(m,m) \asymp \|J_ b f_ m \|_{H^p}^p \ge d^p$, where $d=\inf_{k \in \mathbb{N}} \|J_b f_k\|_{H^p}$, we have that
\[
I(m,m)\ge \frac{C_ 3}{2}\, |\alpha_ m|^2 \,d^p -C_ 2 \delta^2 \, \|\alpha\|^2_{\ell^p}\,4^{-2m},
\]
where $C_ 3$ is another positive constant. Putting the previous estimates in \eqref{Eq-25}, we obtain
\[
\begin{split}
\|U(\alpha)\|_{H^p}^p & \gtrsim 2^{2-p}C_ 3d^p\sum_{m=1}^\infty |\alpha_m|^{p} -C_ 2 2^{3-p}\delta^2 \, \|\alpha\|^2_{\ell^p}\,\sum_{m=1}^{\infty} |\alpha_ m|^{p-2}\,4^{-2m} -C_ 1 \delta^2 \|\alpha\|^p_{\ell^p}
\\
&\ge 2^{2-p}C_ 3d^p\|\alpha\|^p_{\ell^p} -C_ 2 2^{3-p}\delta^2 \, \|\alpha\|^p_{\ell^p} -C_ 1 \delta^2 \|\alpha\|^p_{\ell^p}
\\
& \ge 2^{1-p}C_ 3d^p\|\alpha\|^p_{\ell^p}
\end{split}
\]
after taking  $\delta>0$  small enough. Hence $U$ is bounded below.\\

Let now $2 < p \le 3.$  The preceding estimates together with \eqref{eq: 8} imply that
\begin{align*}
\|U\alpha\|_{H^p}^p &\gtrsim \sum_{m=1}^\infty \int_{\tilde{E}_m}\int_{\Gamma(\xi)}\left|\sum_{j=1}^\infty \alpha_j J_b f_j\right|^{p-2}d\mu d\sigma(\xi)
\\
&\ge \sum_{m=1}^\infty |\alpha_m|^{p-2} I(m,m) -\sum_{m =1}^\infty\sum_{\substack{j = 1 \\ j \ne m}}^\infty |\alpha_j|^{p-2} I(m,j)
\\
&\ge\sum_{m=1}^\infty\left(\frac{C_ 3\,d^p}{2}|\alpha_m|^p-C_ 2 \delta^2 \|\alpha\|_{\ell^p}^p 4^{-2m}\right)-C_ 4\delta^2\|\alpha\|_{\ell^p}^p\sum_{m =1}^\infty\sum_{\substack{j = 1 \\ j \ne m}}^{\infty} 4^{-j-m}
\\
&\gtrsim d^p\,\|\alpha\|_{\ell^p}^p,
\end{align*}
whenever $\delta$ is small enough.\\

 Finally, it remains to deal with the case $1 \le p \le 2.$  By the area theorem, we have
\[
\begin{split}
\|U\alpha\|_{H^p}^p &\asymp  \int_{\mathbb S_n} \left( \int_{\Gamma(\xi)}\left| \sum_{k = 1}^\infty \alpha_k  f_k  \right|^2 |Rb|^2\,dV_{1-n}\right)^{\frac{p}{2}}\, d\sigma(\xi)
\\
&\ge \sum _{m=1}^{\infty} \int_{\tilde{E}_ m} \left( \int_{\Gamma(\xi)}\left| \sum_{k = 1}^\infty \alpha_k  f_k  \right|^2 |Rb|^2\,dV_{1-n}\right)^{\frac{p}{2}}\, d\sigma(\xi).
\end{split}
\]
Applying the $L^2$ triangle inequality, we get
\begin{align*}
&\left( \int_{\Gamma(\xi)}\left| \sum_{k = 1}^\infty \alpha_k  f_k  \right|^2 |Rb|^2\,dV_{1-n}\right)^{\frac{p}{2}}\\
\geq &\left|\left(\int_{\Gamma(\xi)}|\alpha_m f_m|^2 |Rb|^2 dV_{1-n}\right)^{1/2}-\left(\int_{\Gamma(\xi)}\Big|\sum_{\substack{k =1 \\ k \ne m}}^{\infty} \alpha_k f_k\Big|^2 |Rb|^2 dV_{1-n}\right)^{1/2}\right|^p .
\end{align*}
Then, use the $L^p$ triangle inequality to obtain
\begin{align*}
\int_{\tilde{E}_ m}& \left( \int_{\Gamma(\xi)}\left| \sum_{k = 1}^\infty \alpha_k  f_k  \right|^2 |Rb|^2\,dV_{1-n}\right)^{\frac{p}{2}}\, d\sigma(\xi)\\
\geq &\left|\alpha_m \tilde{F}_m(m)^{1/p}-\left(\int_{\tilde{E}_ m}\left(\int_{\Gamma(\xi)}\Big|\sum_{\substack{k =1 \\ k \ne m}}^{\infty} \alpha_k f_k\Big|^2 |Rb|^2 dV_{1-n}\right)^{p/2}d\sigma(\xi)\right)^{1/p}\right|^p .
\end{align*}
This, together with the $\ell^p$ triangle inequality, gives
\[
\begin{split}
\|U\alpha\|_{H^p}^p & \gtrsim \sum_{m=1}^{\infty}\left|\alpha_m \tilde{F}_m(m)^{1/p}-A(m)^{1/p}\right|^p
\\
& \ge \left| \left( \sum_{m=1}^\infty|\alpha_m|^p \tilde{F}_m(m)\right)^{1/p}
-\left( \sum_{m=1}^\infty  A(m)\right)^{1/p}\right|^p,
\end{split}
\]
where
\[
A(m):=\int_{\tilde{E}_ m}\left(\int_{\Gamma(\xi)}\Big |\sum_{\substack{k =1 \\ k \ne m}}^{\infty} \alpha_k f_k\Big|^2 |Rb|^2 dV_{1-n}\right)^{p/2}d\sigma(\xi).
\]
By Lemma \ref{le_functionals-bis} we have $\tilde{F}_ m(k)\le C_ 5 \delta 4^{-k-m}$ for $k\ne m$. Thus, by the estimates obtained in the proof of the case $1\le p\le 2$ in Lemma \ref{bddU}, we have
\[
\begin{split}
\sum_{m=1}^{\infty} A(m) & \le \sum_{m=1}^\infty \Big(\sum_{\substack{k=1\\ k\ne m}}^\infty |\alpha_k|\,\tilde{F}_m(k)^{1/p}\Big)^{p}
\\
& \le C_ 5 \delta  \sum_{m=1}^\infty \left(\sum_{k =1}^\infty ( 4^{-k-m})^{1/p}\right)^p \|\alpha\|_{\ell^p} \le C_ 6 \delta\,\|\alpha \|^p_{\ell^p}.
\end{split}
\]
By the area theorem, we have $\tilde{F}_ m(m)\asymp \|J_ b f_ m\|_{H^p}^p$. Hence, there is a positive constant $C_ 7$ such that $\tilde{F}_ m(m)\ge C_ 7 d^p$, where $d=\inf_{k \in \mathbb{N}} \|J_b f_k\|_{H^p}$.
Therefore we obtain the desired lower bound
\begin{align*}
\|U\alpha\|_{H^p} &\gtrsim  C_ 7^{1/p} d\,\|\alpha\|_{\ell^p} - \big (C_ 6 \delta \big )^{1/p} \|\alpha\|_{\ell^p}
\gtrsim  d \,\|\alpha\|_{\ell^p},
\end{align*}
whenever $\delta$ is small enough.
\end{proof}

Now we are ready to prove our second main result.

\begin{proof}[\textbf{Proof of Theorem \ref{mt2}}]
We observe by the boundedness of point evaluations on $H^p$ that a normalized basis $(g_n)$ of $M$ is locally bounded (consider e.g.~closed balls $B(0, 1-1/n), n = 1, 2,\ldots,$ and recall that $$|f(z)| \lesssim \frac{\|f\|_{H^p}}{(1-|z|^2)^{(n+1)/p}}.)$$ Then by Montel's theorem $(g_n)$ is a normal family and hence there is a subsequence $(g_{n_k})$ converging  to some holomorphic function $g$ uniformly on compact subsets of $\B^n$. Also, $g \in H^p$, by the uniform convergence on $r\mathbb{S}^n$ for every $0 < r < 1$ and the definition of $H^p$-norm. Now take $f_k = \frac{g_{n_k}-g}{\|g_{n_k}-g\|_{H^p}}$ to obtain a normalized sequence $(f_k)$ converging to zero uniformly on compact subsets of $\B^n$. Now by Lemmas \ref{bddU} and \ref{bddbelU}, we can find a subsequence $(f_{n_k})$ such that the operator $U\colon \ell^p \to H^p$ given by \, $$U\alpha = \sum_{k=1}^\infty \alpha_k J_b f_{n_k}$$ is an isomorphism onto its range $J_bN$, where $N = \overline{\textup{span}\{f_{n_k}}\}$. Then the subspace $N$ is isomorphic to $\ell^p$ and in particular the operator $J_b$ is $\ell^2$-singular for $p \ne 2$.
\end{proof}

\begin{remark}

(1) Due to the fact that the sequence spaces $\ell^p$ and $\ell^q$ are totally incomparable for $p\ne q$ whenever $1\le p,q<\infty$, Theorem 1.2 also implies that $J_b\in S_{q}(H^p)$ for $1\le p,q<\infty$, $p \ne q$, and $b\in BMOA(\Bn)$.

(2)  It is known that the standard Bergman spaces $A_{\alpha}^{p}$, $\alpha>-1$, are isomorphic to  $\ell^p$. If $1\le p<\infty$, the strict singularity of the operator $J_b$ on $A_{\alpha}^{p}$ coincides with the compactness, since all strictly singular operators on $\ell^p$  are compact. In particular, we have $J_b\in S_{q}(A_{\alpha}^{p})$ for $1\le p,q<\infty$, $p \ne q$, and $b$ being in the Bloch space.
\end{remark}


\begin{thebibliography}{99}

\bibitem{AB} P. Ahern \and J. Bruna, \emph{Maximal and area integral characterizations of Hardy-Sobolev spaces in the unit ball of $\Cn$}, Rev. Mat. Iberoamericana 4 (1988), 123--153.

\bibitem{Aleman06}
A.~Aleman, \emph{A class of integral operators on spaces of analytic functions},
Topics in Complex Analysis and Operator Theory, Proc.\ Winter School (Antequera, 2006),
pp.\ 3--30.

\bibitem{AC} A. Aleman \and J. Cima, \emph{An integral operator on $H^p$ and Hardy's inequality}, J. Anal. Math. 85 (2001), 157--176.

\bibitem{AS0} {A. Aleman \and A. Siskakis}, \emph{An integral operator on $H^p$},  Complex
Variables 28 (1995), 149--158.

\bibitem{Hu} Z. Hu, \emph{Extended Ces\`{a}ro operators on mixed norm spaces}, Proc. Amer. Math. Soc. 131 (2003), 2171--2179.

\bibitem{Kat} T. Kato, \emph{Perturbation theory for nullity, deficiency and other quantities of linear operators}, J. Anal. Math. 6 (1958), 261--322.

\bibitem{LNST} J. Laitila, P.J. Nieminen, E. Saksman \and H-O. Tylli, \emph{Rigidity of composition operators on the Hardy space $H^p$},  Adv. Math. 319 (2017), 610--629.

\bibitem{M1} S. Miihkinen, \emph{Strict singularity of a Volterra-type integral operator on $H^p$}, Proc. Amer. Math. Soc. 145 (2017), 165--175.

\bibitem{MNST}  S. Miihkinen,  P.J. Nieminen, E. Saksman \and H-O. Tylli, \emph{Structural rigidity of generalised Volterra operators on $H^p$}, Bull. Sci. Math. 148 (2018), 1--13.

\bibitem{MPPW} S. Miihkinen, J. Pau, A. Per\"{a}l\"{a} \and M. Wang, \emph{Volterra type integration operators from Bergman spaces to Hardy spaces}, J. Funct. Anal. (to appear), https://doi.org/10.1016/j.jfa.2020.108564.

\bibitem{Pau2016}
J.~Pau, \emph{Integration operators between Hardy spaces on the unit ball of $\C^n$}, J. Funct. Anal. 270 (2016), 134--176.

\bibitem{P1980} A. Pietsch: \emph{Operator Ideals} (North-Holland, 1980).

\bibitem{Rud} W. Rudin, `Function Theory in the Unit Ball of $\mathbb{C}^n$', Springer, New York, 1980.

\bibitem{Sis} A. Siskakis, \emph{Volterra operators on spaces of analytic functions—a survey}, Proceedings of the First Advanced Course in Operator Theory and Complex Analysis, 51–-68, Univ. Sevilla Secr. Publ., Seville, 2006.

\bibitem{Woj1983}
P.~Wojtaszczyk, \emph{Hardy spaces on the complex ball are isomorphic to {H}ardy spaces on the disc,
 $1 < p < \infty$}, Ann. of Math 118 (1983),  21--34.

\bibitem{ZhuBn} K. Zhu, `{Spaces of holomorphic functions in the unit ball}', Springer-Verlag, New York, 2005.

\end{thebibliography}
\end{document}